\documentclass{amsart}

\usepackage{amssymb,euscript,tikz,units}
\usepackage[colorlinks,citecolor=blue,linkcolor=red]{hyperref}
\usepackage{verbatim}
\usepackage[nameinlink]{cleveref}
\usepackage{colonequals}
\usepackage{tikz}

\newcommand{\M}{\mathcal{M}}
\newcommand{\sM}{\mathcal{M}}
\newcommand{\R}{\mathbb{R}}

\newcommand{\Z}{\mathbb{Z}}

\newcommand{\abs}[1]{\lvert #1 \rvert}

\newcommand{\eps}{\varepsilon}

\def\length{\mathop{\rm length}}
\def\sys{\mathop{\rm sys}}
\def\area{\mathop{\rm Area}}
\def\arcsinh{\mathop{\rm arcsinh}}

\def\Vol{\mathop{\rm Vol}}
\def\dist{\mathop{\rm dist}}
\def\gir{\mathop{\rm girth}}

\theoremstyle{plain}
\newtheorem{theorem}{Theorem}

\newtheorem{proposition}[theorem]{Proposition}
\newtheorem{lemma}[theorem]{Lemma}

\newtheorem{remark}[theorem]{Remark}
\newtheorem{conjecture}[theorem]{Conjecture}

\newcommand{\be}{\begin{equation}}
\newcommand{\ene}{\end{equation}}
\newcommand{\br}{\begin{remark}}
\newcommand{\er}{\end{remark}}
\newcommand{\bl}{\begin{lem}}
\newcommand{\el}{\end{lem}}
\newcommand{\bcor}{\begin{cor}}
\newcommand{\ecor}{\end{cor}}
\newcommand{\bpro}{\begin{pro}}
\newcommand{\epro}{\end{pro}}
\newcommand{\ben}{\begin{enumerate}}
\newcommand{\een}{\end{enumerate}}
\newcommand{\bp}{\begin{proof}}
\newcommand{\ep}{\end{proof}}
\newcommand{\bpo}{\begin{pro}}
\newcommand{\epo}{\end{pro}}
\newcommand{\beq}{\begin{equation*}}
\newcommand{\eeq}{\end{equation*}}
\newcommand{\bear}{\begin{eqnarray}}
\newcommand{\eear}{\end{eqnarray}}
\newcommand{\beqar}{\begin{eqnarray*}}
\newcommand{\eeqar}{\end{eqnarray*}}
\newcommand{\bt}{\begin{theorem}}
\newcommand{\et}{\end{theorem}}
\newcommand{\bex}{\begin{excer}}
\newcommand{\eex}{\end{excer}}

\theoremstyle{definition}

\theoremstyle{remark}
\newtheorem{rem}[theorem]{Remark}
\newtheorem*{rem*}{Remark}
\newtheorem*{pro*}{Proposition}

\newcommand{\RS}{Riemann surface}

\begin{document}

\title{Small eigenvalues of closed Riemann surfaces for large genus }

\author{Yunhui Wu}
\address[Y. ~W. ]{Tsinghua University, Haidian District, Beijing 100084, China}
\email{yunhui\_wu@mail.tsinghua.edu.cn}

\author{Yuhao Xue}
\address[Y. ~X. ]{Tsinghua University, Haidian District, Beijing 100084, China}
\email{xueyh18@mails.tsinghua.edu.cn}

%\date{\today}

\begin{abstract}
In this article we study the asymptotic behavior of small eigenvalues of hyperbolic surfaces for large genus. We show that for any positive integer $k$, as the genus $g$ goes to infinity, the minimum of $k$-th eigenvalues of hyperbolic surfaces over any thick part of moduli space of \RS s of genus $g$ is uniformly comparable to $\frac{1}{g^2}$ in $g$. And the minimum of $ag$-th eigenvalues of hyperbolic surfaces in any thick part of moduli space is bounded above by a uniform constant only depending on $\eps$ and $a$.

In the proof of the upper bound, for any constant $\eps>0$, we will construct a closed hyperbolic surface of genus $g$ in any $\eps$-thick part of moduli space such that it admits a pants decomposition whose curves all have length equal to $\eps$, and the number of separating systole curves in this surface is uniformly comparable to $g$.
\end{abstract}

\maketitle

\section{\textbf{Introduction}}
For a closed \RS \ $X_g$ of genus $g\geq 2$, the complex structure uniquely determines a complete hyperbolic metric on $X_g$. The spectrum of the Laplacian on $X_g$ given by the hyperbolic metric has been a fascinating topic in a number of mathematical fields for a long time. It is well-known that the spectrum of $X_g$ is a discrete subset in $\R^{\geq 0}$ and consists of eigenvalues with finite multiplicity. We enumerate them, counted with multiplicity, in the following increasing order
\[0=\lambda_0(X_g)<\lambda_1(X_g)\leq \lambda_2(X_g) \leq \cdots.\]

Buser \cite{Bus77} showed that for any constant $\eps>0$, there exists a hyperbolic surface $X_g$ of genus $g \geq 2$ such that $\lambda_{2g-3}(X_g)<\eps$.  Schoen, Wolpert and Yau \cite{SWY80} showed that for any integer $i \in [1, 2g-3]$, the $i$-th eigenvalue $\lambda_i(X_g)$ is comparable to a quantity $L_i$ on $X_g$. Here $L_i$ is the minimal possible sum of the lengths of simple closed geodesics in $X_g$ which cut $X_g$ into $i+1$ pieces. Clearly the quantity $L_i\ (1\leq i \leq (2g-3))$ can be arbitrary close to $0$ for certain hyperbolic surfaces. And they \cite{SWY80} also showed that $\lambda_{2g-2}(X_g)>c(g)$ where $c(g)$ is a positive constant only depending on $g$. Then Buser \cite[Theorem 8.1.4]{Bus92} showed that the constant $c(g)$ can be chosen to be uniform (independent of $g$).  Otal and Rosas \cite{OR09} showed that for any analytic Riemannian metric on $X_g$ with curvature $\leq -1$, then the constant $c(g)$ above can be exactly chosen to be $\frac{1}{4}$. And it is known that a complete hyperbolic metric is always analytic. Ballmann, Matthiesen and Mondal \cite{BMM16} showed that the assumption on the real analyticity of the metric in \cite{OR09} can be removed.

Let $\sM_g$ be the moduli space of \RS s of genus $g$. Fix any positive integer $k$, as introduced above, by Schoen-Wolpert-Yau \cite{SWY80} the $k$-th eigenvalue $\lambda_{k}$ can be arbitrarily closed to $0$ near certain boundary of $\sM_g$ for large enough $g$. For any $X=X_g \in \sM_g$, let $h(X)$ be the Cheeger constant of $X \in \sM_g$ (one may see \cite{Che70} or Section \ref{sec-low} for the definition). Mirzakhani \cite[Theorem 4.8]{Mir13} showed that the following limit probability holds:
\[\lim_{g\to \infty} \frac{\Vol_{wp}(\{X\in \sM_g; \ h(X) \geq \frac{\log 2}{2\pi+\log 2} \})}{\Vol_{wp}(\sM_g)}=1\]
where $\Vol_{wp}$ is the Weil-Petersson measure induced from the classical Weil-Petersson metric on $\sM_g$.

Recall that the Cheeger inequality \cite{Che70} tells that the first eigenvalue $\lambda_1(X)\geq \frac{1}{4} h^2(X)$. Hence, Mirzakhani's equation above gives that in a probabilistic sense the first eigenvalue of hyperbolic surfaces is uniform positive as $g \to \infty$. More precisely,
\[\lim_{g\to \infty} \frac{\Vol_{wp}(\{X\in \sM_g; \ \lambda_1(X) \geq \frac{1}{4}\cdot (\frac{\log 2}{2\pi+\log 2})^2 \sim 0.00247 \})}{\Vol_{wp}(\sM_g)}=1.\]
\noindent Recently in \cite{WX21} we improve the above constant $0.00247$ to be $\frac{3}{16}-\eps$ for any $\eps>0$. Which is also independently obtained on later by Lipnowski and Wright in \cite{LW21}.

We refer to the length of a shortest closed geodesic in $X$ as the \emph{systole} of $X$ and denote it by $\sys(X)$.  The systole function $\sys(\cdot):\sM_g \to \R^+$ is continuous, but not smooth as corners appear when it is realized by multiple essential isotropy classes of simple closed curves. Many results on the geometry of moduli space $\sM_g$ especially for large genus $g$ can be stated in terms of the systole function. One may see \cite{BB16, CP12, FKM13, Mirz10, RT13, W-inradius} for recent related topics.

For any positive constant $\eps$, let $\sM_g^{\geq \eps}$ be the $\eps$-thick part of the moduli space $\sM_g$. That is,
\[\sM_g^{\geq \eps}:=\{X \in \sM_g; \ \sys(X) \geq \eps\}.\]

It is known that the set $\sM_g^{\geq \eps}$ is compact for all $\eps>0$, which is due to Mumford \cite{Mumford71}. And it is not hard to see that $\sM_g^{\geq \eps}$ is always nonempty for large enough $g$. Actually Buser-Sarnak \cite{BS94} showed that $\max_{X \in \sM_g}\sys(X)$ is uniformly comparable to $\log{g}$ as $g$ goes to infinity.

For any fixed positive integer $k$, the $k$-th eigenvalue
\[\lambda_k(\cdot): \sM_g \to \R^+\]
defines a continuous and bounded function \cite{Bus92}. As introduced above we know that for large enough $g$,
\[\inf_{X\in \sM_g}  \lambda_k(X) =0.\]

Restricted on the thick part, there exists a constant $c(g,k,\epsilon)>0$, depending on $g,k$ and $\eps$, such that
\[\min_{X\in \sM_g^{\geq \eps}}  \lambda_k(X) \geq c(g, k,\epsilon)>0.\]

In this paper, our main goal is to study the asymptotic behavior of this quantity $\min\limits_{X\in \sM_g^{\geq \eps}}  \lambda_k(X)$ as $g \to \infty$.\\

\noindent \textbf{Notation.} In this paper, for two positive functions $f_1$ and $f_2$ of genus $g$ we say
$$f_1 \succ_g f_2 $$
if there exists a constant $C>0$, independent of $g$, such that
\[\liminf_{g\to \infty}\frac{f_1}{f_2} \geq C.\]
Similarly, we say $$f_1 \prec_g f_2$$ if
\[\limsup_{g\to \infty}\frac{f_1}{f_2} \leq C.\]

We say
$$f_1 \asymp_g f_2$$
if \[f_1 \succ_g f_2 \quad \text{and} \quad f_1 \prec_g f_2.\]
\

Now we are ready to state our results.
\begin{theorem}\label{mt-1}
For any integer $k\geq 1$ and constant $\eps >0$, we have
$$\min_{X\in \M_g^{\geq \eps}} \lambda_k(X) \asymp_g \frac{1}{g^2}.$$
More precisely, there exists a positive constant $\beta(\eps)$, only depending on $\eps$, such that for $g$ large enough,
$$\frac{\eps^2}{16 \pi^2}\cdot \frac{1}{ g^2} \leq \min_{X\in \M_g^{\geq \eps}} \lambda_k(X) \leq (\beta(\eps) \cdot k^2)\cdot \frac{1}{g^2}.$$
Moreover, if the constant $\eps$ satisfies that $0<\eps < 2 \ln(1+\sqrt{2})$, the constant $\beta (\eps)$ above can be chosen to be $c\cdot \eps$ where $c>0$ is a universal constant independent of $\eps$.
\end{theorem}

\begin{rem*}\ben
\item We remark here that Schoen in \cite{Schoen82} showed that for any closed $n$-dimensional $(n\geq 3)$ hyperbolic manifold $M$, $\lambda_1(M)\geq \frac{c_n}{\Vol(M)^2}$ where $c_n>0$ is a constant only depending on $n$. Moreover, the dependence on $\frac{1}{\Vol(M)^2}$ is optimal by Buser \cite{Buser80}.

\item Recently we show in \cite{WX21-eig-1} that for any $X\in \sM_g$ there exists a uniform constant $c>0$ such that $\lambda_1(X)\geq c\cdot\frac{L_1(X)}{g^2}$ where $L_1(X)$ is the minimal possible sum of the lengths of simple closed geodesics in $X$ which cut $X$ into $2$ pieces.

\item From Mantuano \cite{Man04} it is known that $\lambda_k(X)$ is comparable to the $k$-eigenvalue of a discretization of $X$. However, it is not clear to detect the information on the constants depending on $\eps$. In the proof of Theorem \ref{mt-1}, we will use explicit test functions.
\een
\end{rem*}

It is known \cite{BMM16, OR09} that $\lambda_{2g-2}(X_g)>\frac{1}{4}$ for any hyperbolic surface $X_g$. We prove the following uniform upper bound.
\begin{theorem}\label{mt-large}
For any constant $\eps >0$ and $a>0$, we have
$$\min_{X\in \M_g^{\geq \eps}} \lambda_{ag}(X) \leq \frac{1}{4}+a^2\cdot \theta(\eps)$$
where $\theta(\eps)>0$ is a constant only depending on $\eps$ satisfying $\theta(\eps) \to 0 $ as $\eps \to 0$.
\end{theorem}
\begin{rem*}
\ben
\item We are grateful to an anonymous referee for telling us the current statement and its proof on Theorem \ref{mt-large}, where our previous statement is only for $ag=6g-7$.

\item The constant $\eps>0$ in Theorem \ref{mt-1} and \ref{mt-large} could be arbitrarily large.
\een
\end{rem*}

The largest parts of our proofs of Theorem \ref{mt-1} and \ref{mt-large} are to show that they are optimal as $g\to \infty$. We will construct a hyperbolic surface $\mathcal{X}_g \in \M_g^{\geq \eps}$ and show that both the upper bounds in Theorem \ref{mt-1} and \ref{mt-large} will be realized on this hyperbolic surface. The existence of such a hyperbolic surface \ is independently interesting. We formulate it as follows. A curve $\gamma$ is said to be \emph{separating} on $X$ if the complement $X \setminus \gamma$ has at least two components. A pants decomposition $P$ is a collection of $(3g-3)$ disjoint simple closed curves that divide the surface into $(2g-2)$ pants.

\begin{proposition} \label{mt-2}
For any constant $\eps>0$, there exists an integer $g(\eps)>0$, depending on $\eps$, such that for all $g\geq g(\eps)$, there exists a hyperbolic surface $\mathcal{X}_g $ of genus $g$ satisfying that

$(a)$. $\mathcal{X}_g \in \sM_g^{\geq \eps}$.

$(b)$. There exists a pants decomposition $P$ of $\mathcal{X}_g$ such that for every curve $\gamma\in P$, the length $\ell(\gamma)$ of $\gamma$ is equal to $\eps$.

$(c)$. The number $\#\{\gamma \in P; \ \text{$\gamma$ is separating on $\mathcal{X}_g$}\}\asymp_g g$.
\end{proposition}

For small enough constant $\eps>0$, by using the Collar Lemma it is not hard to construct such a surface. Petri \cite[Proposition 6.2]{Petri18} constructed a closed surface satisfying (b) and all other simple curves that are not part of the given pants decomposition are ``long enough". One may also see \cite{BRafi18} for related constructions (in more complicated settings). Here $\eps>0$ can be arbitrarily large. The new insight of Proposition \ref{mt-2} is that the number of separating systolic curves on $\mathcal{X}_g$ is uniformly comparable to $g$. By construction in Section \ref{sec-ex}, the twist parameters on the curves in $P$ can be chosen to be arbitrary. So actually we construct a subset in $\sM_g^{\geq \eps}$ of dimension $(3g-3)$ consisting of closed hyperbolic surfaces of genus $g$ satisfying Proposition \ref{mt-2}.

\begin{rem*}
It is interesting to study the asymptotic behavior of the quantity $\sup \limits_{X\in \sM_g}  \lambda_k(X)$ as $g \to \infty$. One may see \cite{BM01,Bus84, BBD88,Mon15} for related details. It is known that $\limsup\limits_{g\to \infty}\sup\limits_{X\in \sM_g}  \lambda_1(X)\leq \frac{1}{4}$. One interesting conjecture is as following which is related to the famous Selberg conjecture \cite{Sarnak-95}.
\end{rem*}

\begin{conjecture}[\cite{BM01}]
 $\lim\limits_{g\to \infty}\sup\limits_{X\in \sM_g}  \lambda_1(X)=\frac{1}{4}$. 
\end{conjecture}

Recently this conjecture is confirmed by Hide and Magee in \cite{HM21}.

\subsection*{Plan of the paper.} Section \ref{np} will provide some necessary background and basic properties on hyperbolic surfaces. In Section \ref{sec-low} we will provide a proof for the lower bound of Theorem \ref{mt-1}. The construction of hyperbolic surface will be discussed in Section \ref{sec-ex}, that is to show Proposition \ref{mt-2}. In Section \ref{sec-upp} we will complete the proofs of Theorem \ref{mt-1} and \ref{mt-large}.

\subsection*{Acknowledgement}
The authors would like to greatly thank Hugo Parlier and Kasra Rafi for the invaluable discussions on Proposition \ref{mt-2}. We are also very grateful to an anonymous referee for his/her helpful comments and suggestions, especially for improving the statement of Theorem \ref{mt-2} and providing its proof, and also for pointing out the related works by Buser \cite{Buser80}, Erd\"os--Sachs \cite{ES63}, Schoen \cite{Schoen82}, and so on. Which makes the paper more clear. Both authors are supported by the NSFC grant No. $12171263$.

%%%%%%%%%%%%%%%%%%%%.%%%%%%%%%%%%%%%%
\section{Preliminaries} \label{np}
In this section we will set up the notations and provide some necessary background on two-dimensional hyperbolic geometry, regular graphs and spectrum theory of the Laplace operator.

%%%%%%%%%
\subsection{Hyperbolic surfaces} Let $X_g$ be a closed hyperbolic surface of genus $g\geq 2$ and $\gamma \subset X_g$ be a homotopically non-trivial loop. There always exists a unique closed geodesic, still denoted by $\gamma$, freely homotopic to this loop. The Collar Lemma says that it has a tubular neighborhood which is a topological cylinder with a standard hyperbolic metric. And the width of this cylinder, only depending on the length of $\gamma$, goes to infinity as the length of $\gamma$ goes to $0$. This was first observed by Keen in \cite{Kee74} and then improved by many people. We use the following version \cite[Theorem 4.1.1]{Bus92}.

\begin{lemma}[Collar lemma]\label{collar}
Let $\gamma_1 , \gamma_2, ..., \gamma_m$ be disjoint simple closed geodesics on a closed hyperbolic surface $X_g$, and $\ell(\gamma_i)$ be the length of $\gamma_i$. Then $m\leq 3g-3$ and we can define the collar of $\gamma_i$ by
$$K(\gamma_i)=\{x\in X_g; \ \dist(x,\gamma_i)\leq w(\gamma_i)\}$$
where
$$w(\gamma_i)=\mathop{\rm arcsinh} \frac{1}{\sinh \frac{1}{2}l(\gamma_i)}$$
is the half width of the collar.

Then the collars are pairwise disjoint for $i=1,...,m$. Each $K(\gamma_i)$ is isomorphic to a cylinder $(\rho,t)\in [-w(\gamma_i),w(\gamma_i)] \times \mathbb S ^1$, where $\mathbb S ^1 = \R / \Z$, with the metric
$$ds^2=d\rho^2 + \ell(\gamma_i)^2 \cosh^2\rho dt^2.$$
And for a point $(\rho,t)$, the point $(0,t)$ is its projection on the geodesic $\gamma_i$, $\abs{\rho}$ is the distance to $\gamma_i$, $t$ is the coordinate on $\gamma_i \cong \mathbb S ^1$.
\end{lemma}
\noindent As the length $\ell(\gamma)$ of the central closed geodesic goes to $0$, the width $w$ tends to infinity; and if the length $\ell(\gamma)$ goes to infinity, the width $w$ tends to zero. In this paper, we mainly deal with the case that $\ell(\gamma)$ is large.

%%%%%%%%%%%%%%%%%%%%%%%%

\subsection{Pants} \emph{A pair of pants} $\mathcal{P}$ is a compact hyperbolic surface of genus $0$ with $3$ totally geodesic boundary components, each homeomorphic to a circle. The hyperbolic metric is uniquely determined by the lengths of the three boundary closed geodesics.

For any constant $\eps>0$, let $\mathcal{P}_{\eps}$ be the pair of pants whose boundary curves all have length equal to $\eps$. Let $d$ be a shortest path between two different boundary curves. We denote its length also by $d$. The curve $d$ is perpendicular to the boundary curves. Let $\tau$ be a shortest path with both end points on a given boundary curve and assume that $\tau$ is not homotopic to a piece of the boundary curve (with end points fixed). Then this curve $\tau$ is clearly also perpendicular to the given boundary curve at both end points. We denote the length of $\tau$ also by $\tau$.
\begin{figure}[ht]
\begin{center}
\begin{tikzpicture}[scale=1]

\draw[rotate=60] (-1,2) ellipse [x radius=1, y radius=0.3];
\draw[rotate=-60] (1,2) ellipse [x radius=1, y radius=0.3];
\draw[rotate=0] (0,-3) ellipse [x radius=1, y radius=0.3];

\draw (195:2.828) ..controls (-1.3,-1.5) .. (-1,-3);
\draw (-15:2.828) ..controls (1.3,-1.5) .. (1,-3);
\draw (150:2) ..controls (0,0.3) .. (30:2);

\draw (0.2,-2.7) ..controls (0.2,-1).. (0,0.47);
\draw[dashed] (-0.2,-3.3) ..controls (-0.2,-1).. (0,0.47);

\draw (5:2.7) node {$\eps$};
\draw (175:2.7) node {$\eps$};
\draw (0,-3.5) node {$\eps$};

\draw (-45:2.3) node {$d$};
\draw (225:2.3) node {$d$};
\draw (0,0.8) node {$d$};

\draw (0.3,-1) node {$\tau$};

\end{tikzpicture}
\end{center}
\caption{}\label{pic2}
\end{figure}

The three geodesics $d$ divide the pants into two equal right-angled hexagons. Let $h$ be the shortest path between a pair of opposite sides in the right-angled hexagon. We denote the length of $h$ also by $h$. By symmetry, clearly $h$ intersects the both boundary curves at the midpoints. And
\[\tau=2h.\]
\begin{figure}[ht]
\begin{center}
\begin{tikzpicture}[scale=1]

\draw (195:2.828) ..controls (-2.1,0) .. (150:2);
\draw (-15:2.828) ..controls (2.1,0) .. (30:2);
\draw (-1,-3) ..controls (0,-2.8) .. (1,-3);

\draw (195:2.828) ..controls (-1.7,-1.7) .. (-1,-3);
\draw (-15:2.828) ..controls (1.7,-1.7) .. (1,-3);
\draw (150:2) ..controls (0,0.7) .. (30:2);

\draw (0,-2.85) -- (0,0.77);

\draw (5:2.4) node {$\eps/2$};
\draw (175:2.4) node {$\eps/2$};
\draw (0,-3.1) node {$\eps/2$};

\draw (-45:2.7) node {$d$};
\draw (225:2.7) node {$d$};
\draw (0,1) node {$d$};

\draw (0.2,-1) node {$h$};

\end{tikzpicture}
\end{center}
\caption{}\label{pic3}
\end{figure}

For right-angled pentagons and hexagons, by formulas in \cite[Page 454]{Bus92} we have
$$\cosh \frac{\eps}{2} = \sinh^2 {\frac{\eps}{2}} \cosh {d} - \cosh^2 \frac{\eps}{2}$$
and
$$\cosh \frac{\eps}{2} = \sinh h \sinh \frac{\eps}{4}.$$

Thus,
\bear \label{d-ex}
d=2\arcsinh \frac{1}{2\sinh \frac{\eps}{4}}
\eear
and
\bear \label{tau-ex}
\tau=2h=2\arcsinh \frac{\cosh \frac{\eps}{2}}{\sinh \frac{\eps}{4}}.
\eear

Thus, for large $\eps$, we have
\bear \label{d-in}
\lim_{\eps \to \infty} \frac{d}{2 e^{-\frac{\eps}{4}}}=1.
\eear

And we always have
\bear \label{tau-in}
\tau > \frac{1}{2}\eps.
\eear
%%%%%%%%%%%%%%%%%

\subsection{Trivalent graphs}\label{subs-tg} A \emph{trivalent graph} is a finite $3$-regular connected graph. That is, every vertex has three emanating edges. Every compact hyperbolic surface of genus $g \geq 2$ has a pants decomposition of $(2g-2)$ pairs of pants. If we associate each pair of pants to be a vertex, and an edge between two vertices if two pairs of pants have a common boundary curve, then this associates each compact hyperbolic surface of genus $g \geq 2$ to a trivalent graph with $(2g-2)$ vertices of $(3g-3)$ edges. Actually the opposite way is also true. One may see \cite[Section 3.5]{Bus92} for more details.

Recall that the \emph{girth} of a connected graph is the length of the shortest non-trivial loop in this graph where each edge has length equal to $1$. Actually more general, an $n$-regular graph is a graph such that all vertices have degree $n$. A trivalent graph is a $3$-regular graph. Let $U(n,E,w)$ be the number of unlabeled $n$-regular graphs with $E$ edges with girth at least $w$. It is well-known that for any fixed $n>0$ and arbitrarily large $w>0$, $U(n,E,w)>0$ for large enough $E$. For example one may see \cite{ES63} by Erd\"os--Sachs. In particular, for large enough $g$, there always exist trivalent graphs with edges $(3g-3)$ of arbitrary large girth. In Section \ref{sec-ex} we will use this fact to construct the hyperbolic surfaces in Proposition \ref{mt-2}.

Actually more precisely, by Bollob$\acute{\text{a}}$s \cite{Bol82} and Wormald \cite{Wor81} one may have
\[U(n,E,w) \sim_E \exp (-\sum_{i=1}^{w-1} \frac{(n-1)^i}{2i}) \frac{(2E)!}{2^E E! V! (n!)^V}\]
where $V = 2E/n$ is the number of vertices and $f_1 \sim_E f_2$ means $\lim_{E\to \infty}\frac{f_1(E)}{f_2(E)}=1$.

%%%%%%%%%%%%%%

\subsection{Spectrum} Let $X_g$ be a closed \RS \ of genus $g\geq 2$ which corresponding to a hyperbolic metric on $X_g$. Let $\Delta$ be the Laplacian with respect to this metric. A number $\lambda$ is called an \emph{eigenvalue} if $\Delta f+ \lambda \cdot f=0$ on $X_g$. And the corresponding function $f$ is called an \emph{eigenfunction}. It is known that the set of eigenvalues is an infinite sequence of non-negative numbers
\[0=\lambda_0(X_g)<\lambda_1(X_g)\leq \lambda_2(X_g) \leq \cdots.\]
Let $\{f_i\}_{i \geq 0}$ be its corresponding orthonormal sequence of eigenfunctions. Clearly $f_0$ is the constant function $\frac{1}{4\pi(g-1)}$. The mini-max principle tells that for any integer $k\geq 0$,
\bear \label{lamb-k}
\lambda_k(X_g)=\inf_{}\{ \frac{\int_{X_g}|\nabla f|^2}{\int_{X_g}f^2}; \ f \in H^1(X_g) \ \text{and} \ \int_{X_g}f\cdot f_i=0 \ \forall i \in [0, k-1]\}
\eear
where $H^1(X_g)$ is the completion of the space of smooth functions on $X_g$. One may see \cite{Chavel} for details.

%%%%%%%%%%%%%%%%%%%%%%%%%%%%%%%%%%%%%%%%%%%%%%%%%%%%%%%%%%%%%%%%%%%%%%%%%%%%%%%%%%%%

\section{Lower Bound}\label{sec-low}

In this section, we will use the Cheeger inequality in \cite{Che70} to give the lower bound in Theorem \ref{mt-1}. More precisely, we will show that
\begin{proposition} \label{low}
For any constant $\eps >0$, then for large enough $g$ we have,
\[\min_{X\in \M_g^{\geq \eps}} \lambda_1(X)\geq \frac{\eps^2}{16 \pi^2}\cdot \frac{1}{ g^2}.\]
\end{proposition}

First we recall the Cheeger inequality. Let $X_g$ be a closed hyperbolic surface of genus $g \ (g\geq 2)$ and $\Omega\subset X_g$ be a domain with smooth boundary. The \emph{first eigenvalue $\sigma_1(\Omega)$ of Neumann type} on $\Omega$ is defined as
$$\sigma_1(\Omega):=\inf_{\int_\Omega f=0} \frac{\int_\Omega \abs{\nabla f}^2 }{\int_\Omega f^2}$$
where $\nabla$ is the hyperbolic gradient in the sense of the hyperbolic metric corresponding to the complex structure of $X_g$.

The \emph{Cheeger isoperimetric constant} $h(\Omega)$ is defined as
$$h(\Omega):= \inf \frac{\length(\Gamma)}{\min \{\area(A_1),\area(A_2)\}}$$
where the infimum is taken over all smooth curves $\Gamma$ which divide $\Omega$ into two pieces $A_1$ and $A_2$.

\begin{lemma}[Cheeger inequality, \cite{Che70}]\label{Cheeger}
Then
$$\sigma_1(\Omega) \geq \frac{1}{4}h^2(\Omega).$$
\end{lemma}

We will apply this Cheeger inequality to prove Proposition \ref{low}.

The following two elementary isoperimetric inequalities are well-known. For completeness, we provide proofs here.
\begin{lemma}\label{iso 1}
Assume that $\Omega \subset X_g$ is homeomorphic to an open disk. Then
$$\area(\Omega)\leq \length(\partial\Omega).$$
\end{lemma}

\begin{proof}
Since $\Omega \subset X_g$ is an open disk, one may regard $\Omega$ as a disk in the upper half plane $\mathbb{H}$ endowed with the standard hyperbolic metric $ds^2 =\frac{dx^2+dy^2}{y^2}$. For a smooth function $f(x,y)$ on $\Omega$, direct computations give that
$$\nabla f= y^2\frac{\partial f}{\partial x}\frac{\partial }{\partial x} + y^2\frac{\partial f}{\partial y}\frac{\partial }{\partial y},$$
$$|\nabla f|^2= y^2(|\frac{\partial f}{\partial x}|^2 + |\frac{\partial f}{\partial y}|^2),$$
$$\Delta f=y^2(\frac{\partial^2 f}{\partial x^2} + \frac{\partial^2 f}{\partial y^2})$$

By Stokes' Theorem and the Cauchy-Schwarz inequality, we have
\bear \label{sto}
\int_\Omega \Delta f &=& \int_{\partial\Omega}\nabla f \cdot \vec n \\
&\leq& \int_{\partial\Omega}|\nabla f|. \nonumber
\eear

Set $f(x,y)=-\ln y$. It is clear that
\[\Delta f=1 \quad \emph{and} \quad |\nabla f|=1.\]

Plug the two equations above into \eqref{sto}, we get
$$\int_\Omega dV \leq \int_{\partial\Omega} ds.$$

That is, $$\area(\Omega)\leq \length(\partial\Omega).$$

The proof is complete.
\end{proof}

\begin{lemma}\label{iso 2}
Assume that a domain $\Omega\subset X_g$ is topologically a cylinder. Then
$$\area(\Omega)\leq \length(\partial\Omega).$$
\end{lemma}

\begin{proof}
Since $\Omega\subset X_g$ is topologically a cylinder, the boundary $\partial\Omega$ consists of two homotopic simple closed loops. We split the proof into two cases.

Case-1. The two boundary closed loops of $\Omega$ are homotopic to a point on $X_g$.

For this case, one may assume that one boundary closed loop bounds a disk $D$ such that
\[\Omega\subset D \quad \text{and} \quad \partial D \subset \partial \Omega.\]

Then by Lemma \ref{iso 1} we have,
\beqar
\area(\Omega) &\leq& \area(D) \\
&\leq& \length(\partial D)\\
&\leq& \length(\partial\Omega).
\eeqar
\

Case-2. The two boundary closed loops of $\Omega$ are not homotopic to a point.

For this case, one may assume that $\gamma$ is the unique closed geodesic representing a component of $\partial\Omega$.
Set $$l=\length(\gamma).$$

Let $\R\times \mathbb S^1$ be the infinite cylinder of parameters $(\rho,t)$ endowed with the hyperbolic metric
$$ds^2=d\rho^2 + l^2 \cosh^2\rho dt^2.$$
This is a hyperbolic cylinder with infinite width whose unique closed geodesic has length equal to $l$. By assumption one may assume that $\Omega$ is a subset of this infinite hyperbolic cylinder $\R\times \mathbb S^1$. For a smooth function $f(\rho,t)$ on $\Omega$, direct computations show that
$$\nabla f= \frac{\partial f}{\partial \rho}\frac{\partial }{\partial \rho} + \frac{1}{l^2 \cosh^2\rho}\frac{\partial f}{\partial t}\frac{\partial }{\partial t},$$
$$|\nabla f|^2= |\frac{\partial f}{\partial \rho}|^2 + \frac{1}{l^2 \cosh^2\rho}|\frac{\partial f}{\partial t}|^2,$$
$$\Delta f=\frac{\partial^2 f}{\partial \rho^2} + \tanh {\rho} \cdot \frac{\partial f}{\partial \rho} + \frac{1}{l^2 \cosh^2\rho}\frac{\partial^2 f}{\partial t^2}.$$

Similar as in the proof of Lemma \ref{iso 1} we have
\bear \label{sto-2}
\int_\Omega \Delta f &=& \int_{\partial\Omega}\nabla f \cdot \vec n \\
&\leq& \int_{\partial\Omega}|\nabla f|. \nonumber
\eear

Set $f(\rho,t)=\ln\cosh\rho$. It is clear that
\[\Delta f=1 \quad \emph{and} \quad |\nabla f|=|\tanh\rho|.\]

Plug these two equations above into \eqref{sto-2}, we get
\beqar
\int_\Omega dV &\leq& \int_{\partial\Omega} |\tanh\rho| ds \\
&\le& \int_{\partial\Omega} ds.
\eeqar

That is, $$\area(\Omega)\leq \length(\partial\Omega).$$

The proof is complete.
\end{proof}

Now we are ready to prove Proposition \ref{low}.
\bp [Proof of Proposition \ref{low}]
Let $X_g \in \M_g^{\geq \eps}$. By Cheeger's inequality it suffices to provide a lower bound for the Cheeger isoperimetric constant $h(X_g)$.

Let $\Gamma$ be a set of smooth curves dividing $X_g$ into two disjoint pieces $A_1$ and $A_2$. Then $\Gamma$ must be one of the following three cases:

$(a)$. $\Gamma$ contains a simple closed curve bounding a disk $D$ in $X_g$.

$(b)$. $\Gamma$ contains two simple closed curves $\tau$ and $\gamma$ which are homotopic to each other in $X_g$.

$(c)$. $\Gamma$ is not of type (a) and (b). That is, no two pairwise simple closed curves in $\Gamma$ are homotopic, and no simple closed curve in $\Gamma$ is homotopically trivial. In particular, $\Gamma$ contains at least one nontrivial closed curve.\\

If $\Gamma$ is of type (a), by Lemma \ref{iso 1} we have
$$\length(\Gamma)\geq \area(D).$$

Thus,
\beqar
\frac{\length(\Gamma)}{\min \{\area(A_1),\area(A_2)\}} \geq  \frac{\length(\Gamma)}{\area(D)} \geq 1.
\eeqar
\

If $\Gamma$ is of type (b), then the two curves $\tau$ and $\gamma$ bounds a cylinder $\Omega$ in $X_g$. By Lemma \ref{iso 2} we have
\beqar
\length(\Gamma) &\geq& \length(\gamma)+\length(\tau) \\
&=& \length(\partial\Omega)\\
&\geq& \area(\Omega).
\eeqar

Thus,
\beqar
\frac{\length(\Gamma)}{\min \{\area(A_1),\area(A_2)\}} \geq  \frac{\length(\Gamma)}{\area(D)} \geq 1.
\eeqar
\

If $\Gamma$ is of type (c), since $\Gamma$ contains at least one nontrivial closed curve,
$$\length (\Gamma) \geq \sys (X_g) \geq \eps.$$

It is clear that
\beqar
\min \{\area(A_1),\area(A_2)\} &\leq& \frac{1}{2} \area (X_g) \\
&= & 2\pi (g-1).
\eeqar

Thus,
$$\frac{\length(\Gamma)}{\min \{\area(A_1),\area(A_2)\}} \geq \frac{\eps}{2\pi (g-1)}.$$
\newline

From these three cases above, the Cheeger isoperimetric constant $h(X_g)$ satisfies that
\beqar
h(X_g) &=& \inf \frac{\length(\Gamma)}{\min \{\area(A_1),\area(A_2)\}}\\
&\geq& \min \{1, \frac{\eps}{2\pi(g-1)}\}.
\eeqar

By Lemma \ref{Cheeger}, we have
\beqar
\lambda_1(X_g) &\geq& \frac{h^2(X_g)}{4} \\
&\geq& \min \{\frac{1}{4}, \frac{\eps^2}{16\pi^2(g-1)^2}\}.
\eeqar

For large enough $g$, clearly we get
\[\lambda_1(X_g)\geq \frac{\eps^2}{16\pi^2}\cdot \frac{1}{g^2}.\]

Then conclusion follows.
\ep

%%%%%%%%%%%%%%%%%%%%%%%%%%%%%%%%%%%%%%%%%%%%%%%%%%%%%%%%%%%%%%%%%%%%%%%%%%%%%%%%%%%%

\section{Proof of Proposition \ref{mt-2}}\label{sec-ex}
In this section we will construct a hyperbolic surface satisfying Proposition \ref{mt-2}.

For a trivalent graph $G$ with $V$ vertices of $E$ edges. It is clear that
\[3V= 2E.\]
So every trivalent graph must have an even number of vertices. For any $\eps>0$, we let $$d(\eps)=2\arcsinh \frac{1}{2\sinh \frac{\eps}{4}}$$ be the constant in \eqref{d-ex}. One may assume $W(\eps)>0$ is a constant, only depending on $\eps$, such that
\bear \label{w-s-e}
W(\eps) \cdot d(\eps) \geq 2\eps.
\eear
Clearly as $\eps$ goes to infinity, $W(\eps)$ also goes to infinity. As introduced in Subsection \ref{subs-tg} there always exists an even integer $V_0=V_0(\eps)>0$, only depending on $\eps$, such that for all even integer $V\geq V_0$, there exists a trivalent graph with $V$ vertices, denoted by  $G(V)$, such that the girth $\gir(G(V))$ of $G(V)$ satisfies that
\bear \label{gir-in}
\gir(G(V))\geq W(\eps).
\eear
In particular,
\[\gir(G(V_0))\geq W(\eps).\]

For large enough $g\ (g>>V_0)$, there exists two numbers $g_0$ and $V_1$ such that
\bear \label{ver-eq}
2g-2=g_0(V_0+2)+(V_1+2)-2
\eear
where $V_1$ satisfies
\bear \label{rem-eq}
V_0\leq V_1 \leq 2V_0+2.
\eear

In particular, by \eqref{gir-in} there exists a trivalent graph $G(V_1)$ with vertices $V_1$ such that
\[\gir(G(V_1))\geq W(\eps).\]

We put $g_0$ graphs $G(V_0)$ and a graph $G(V_1)$ from left to right and add one edge between any two consecutive graphs. Then we get a new trivalent graph, denoted by $G_g$. One may see this graph $G_g$ as in figure \ref{pic4}.
\begin{figure}[ht]
\begin{center}
\begin{tikzpicture}[scale=1]

\draw (1,0)-- ++(-1,0) ++(-1,0)-- ++(-1,0) ++(-1,0)-- ++(-1,0);
\draw[dashed] (0,0.5)-- ++(-1,0) ++(-1,0)-- ++(-1,0);
\draw[dashed] (0,-0.5)-- ++(-1,0) ++(-1,0)-- ++(-1,0);
\draw (0,-0.5)-- ++(0,1) (-1,-0.5)-- ++(0,1) (-2,-0.5)-- ++(0,1) (-3,-0.5)-- ++(0,1) (-4,-0.5)-- ++(0,1);
\draw[dashed] (-4,0.5) ..controls (-5.5,0.5) and (-5.5,-0.5).. (-4,-0.5);

\draw[dashed] (1,0)--++(2,0);
\draw (3,0)--++(1,0) (4,-0.5)--++(0,1) ++(1,-1)--++(0,1) ++(0,-0.5)--++(1,0) ++(0,-0.5)--++(0,1);
\draw[dashed] (4,0.5)--++(1,0) (4,-0.5)--++(1,0);
\draw[dashed] (6,0.5) ..controls (8,0.5) and (8,-0.5).. (6,-0.5);

\filldraw (0,0)circle[radius=0.03] (1,0)circle[radius=0.03] (3,0)circle[radius=0.03] (4,0)circle[radius=0.03] (5,0)circle[radius=0.03] (6,0)circle[radius=0.03] (-1,0)circle[radius=0.03] (-2,0)circle[radius=0.03] (-3,0)circle[radius=0.03] (-4,0)circle[radius=0.03];
\filldraw (0,0.5)circle[radius=0.03] (-1,0.5)circle[radius=0.03] (-2,0.5)circle[radius=0.03] (-3,0.5)circle[radius=0.03] (-4,0.5)circle[radius=0.03] (4,0.5)circle[radius=0.03] (5,0.5)circle[radius=0.03] (6,0.5)circle[radius=0.03];
\filldraw (0,-0.5)circle[radius=0.03] (-1,-0.5)circle[radius=0.03] (-2,-0.5)circle[radius=0.03] (-3,-0.5)circle[radius=0.03] (-4,-0.5)circle[radius=0.03] (4,-0.5)circle[radius=0.03] (5,-0.5)circle[radius=0.03] (6,-0.5)circle[radius=0.03];

\draw (-0.5,0) node {$G(V_0)$}  (-2.5,0) node {$G(V_0)$}  (-4.5,0) node {$G(V_0)$}  (4.5,0) node {$G(V_0)$}  (6.5,0) node {$G(V_1)$};

\end{tikzpicture}
\end{center}
\caption{Graph $G_g$}\label{pic4}
\end{figure}
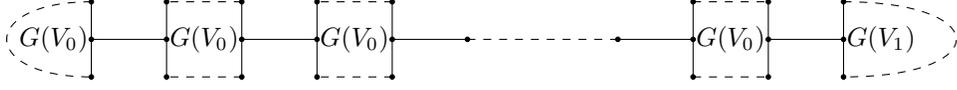

The number of vertices of $G_g$ is $g_0(V_0 +2) + (V_1 +2) -2$. By \eqref{ver-eq} we know that this graph $G_g$ has $(2g-2)$ vertices. And it is clear that any non-trivial loop in $G_g$ can only happen in either one $G(V_0)$ or $G(V_1)$. Thus, the girth satisfies
\bear \label{gir-in-g}
\gir(G_g)\geq W(\eps).
\eear

\begin{rem}\label{r-e-small}
If $\eps$ is small enough, we may choose a special graph $G_g$ as shown in figure \ref{pic5}. In this case, the girth $W(\eps)=1$.
\begin{figure}[ht]
\begin{center}
\begin{tikzpicture}[scale=1]

\draw (1,0)-- ++(-1,0) ++(-1,0)-- ++(-1,0) ++(-1,0)-- ++(-1,0);
\draw[dashed] (1,0)--++(2,0);
\draw (3,0)--++(1,0) ++(1,0)--++(1,0);

\draw (0,0) ..controls(-0.5,0.3).. (-1,0) (0,0) ..controls(-0.5,-0.3).. (-1,0);
\draw (-2,0) ..controls(-2.5,0.3).. (-3,0) (-2,0) ..controls(-2.5,-0.3).. (-3,0);
\draw (4,0) ..controls(4.5,0.3).. (5,0) (4,0) ..controls(4.5,-0.3).. (5,0);

\draw (-4,0) ..controls(-5,0.8)and(-5,-0.8).. (-4,0);
\draw (6,0) ..controls(7,0.8)and(7,-0.8).. (6,0);

\filldraw (0,0)circle[radius=0.03] (1,0)circle[radius=0.03] (3,0)circle[radius=0.03] (4,0)circle[radius=0.03] (5,0)circle[radius=0.03] (6,0)circle[radius=0.03] (-1,0)circle[radius=0.03] (-2,0)circle[radius=0.03] (-3,0)circle[radius=0.03] (-4,0)circle[radius=0.03];

\end{tikzpicture}
\end{center}
\caption{}\label{pic5}
\end{figure}
\end{rem}

The surface $\mathcal{X}_g$ of genus $g$ in Proposition \ref{mt-2} is constructed as follows.

Recall that $\mathcal{P}_{\eps}$ is a pair of pants whose boundary closed geodesics all have length equal to $\eps$. We replace each vertex in the trivalent graph $G_g$ by $\mathcal{P}_{\eps}$ and glue two pair of pants along one boundary loop if they are connected by one edge. In this way, we get a hyperbolic surface $\mathcal{X}_g$ of genus $g$. For examples: for large $\eps$ and graph as showed in figure \ref{pic4}, the surface $\mathcal{X}_g$ looks like figure \ref{pic6}; and for small $\eps$ and the graph as showed in figure \ref{pic5}, the surface $\mathcal{X}_g$ looks like figure \ref{pic7}.

\begin{figure}[ht]
\begin{center}
\begin{tikzpicture}[scale=0.75]

\draw (0,0) +(0,0.3) ..controls+(-1.4,0).. +(-1.4,1)  +(0,-0.3) ..controls+(-1.4,0).. +(-1.4,-1)
+(-2,1)--+(-2,-1)
+(-2,1) ..controls+(0.3,0.1).. +(-1.4,1)  +(-2,1) ..controls+(0.3,-0.1).. +(-1.4,1)
+(-2,-1) ..controls+(0.3,0.1).. +(-1.4,-1)  +(-2,-1) ..controls+(0.3,-0.1).. +(-1.4,-1);
\draw[red,very thick] (0,0) +(0,0.3) ..controls+(0.1,-0.3).. +(0,-0.3)  +(0,0.3) ..controls+(-0.1,-0.3).. +(0,-0.3);

\draw (-6,0) +(0,0.3) ..controls+(1.4,0).. +(1.4,1)  +(0,-0.3) ..controls+(1.4,0).. +(1.4,-1)
+(2,1)--+(2,-1)
+(2,1) ..controls+(-0.3,0.1).. +(1.4,1)  +(2,1) ..controls+(-0.3,-0.1).. +(1.4,1)
+(2,-1) ..controls+(-0.3,-0.1).. +(1.4,-1)  +(2,-1) ..controls+(-0.3,0.1).. +(1.4,-1);
\draw[red,very thick] (-6,0) +(0,0.3) ..controls+(0.1,-0.3).. +(0,-0.3)  +(0,0.3) ..controls+(-0.1,-0.3).. +(0,-0.3);

\draw (-6,0) +(0,0.3) ..controls+(-1.4,0).. +(-1.4,1)  +(0,-0.3) ..controls+(-1.4,0).. +(-1.4,-1)
+(-2,1)--+(-2,-1)
+(-2,1) ..controls+(0.3,0.1).. +(-1.4,1)  +(-2,1) ..controls+(0.3,-0.1).. +(-1.4,1)
+(-2,-1) ..controls+(0.3,0.1).. +(-1.4,-1)  +(-2,-1) ..controls+(0.3,-0.1).. +(-1.4,-1);

\draw (2,0) +(0,0.3) ..controls+(1.4,0).. +(1.4,1)  +(0,-0.3) ..controls+(1.4,0).. +(1.4,-1)
+(2,1)--+(2,-1)
+(2,1) ..controls+(-0.3,0.1).. +(1.4,1)  +(2,1) ..controls+(-0.3,-0.1).. +(1.4,1)
+(2,-1) ..controls+(-0.3,-0.1).. +(1.4,-1)  +(2,-1) ..controls+(-0.3,0.1).. +(1.4,-1);
\draw[red,very thick] (2,0) +(0,0.3) ..controls+(0.1,-0.3).. +(0,-0.3)  +(0,0.3) ..controls+(-0.1,-0.3).. +(0,-0.3);

\draw (0,0) +(0,0.3)--+(0.1,0.3) +(0,-0.3)--+(0.1,-0.3);
\draw (2,0) +(0,0.3)--+(-0.1,0.3) +(0,-0.3)--+(-0.1,-0.3);

\draw[dashed] (0.2,0)--(1.8,0);

\draw[dashed] (-3,1) +(1.3,0.1)..controls+(-1.3,0.4).. +(-1.3,0.1);
\draw[dashed] (-3,-1) +(1.3,-0.1)..controls+(-1.3,-0.4).. +(-1.3,-0.1);

\draw[dashed] (-8,0) +(0.3,1.1)..controls+(-3,0)and+(-3,0)..+(0.3,-1.1);

\draw[dashed] (4,0) +(-0.3,1.1)..controls+(3,0)and+(3,0)..+(-0.3,-1.1);

\draw (-3,0) node {$\Sigma(G(V_0))$};
\draw (-9,0) node {$\Sigma(G(V_0))$};
\draw (5,0) node {$\Sigma(G(V_1))$};

\end{tikzpicture}
\end{center}
\caption{Surface $\mathcal{X}_g$} \label{pic6}
\end{figure}

\begin{figure}[h]
\begin{center}
\begin{tikzpicture}[scale=1]

\draw (0,0.3) ..controls +(-0.5,0)and +(0.5,0).. +(-1,0.5)  (0,-0.3) ..controls +(-0.5,0)and +(0.5,0).. +(-1,-0.5);
\draw (-2,0.3) ..controls +(0.5,0)and +(-0.5,0).. +(1,0.5)  (-2,-0.3) ..controls +(0.5,0)and +(-0.5,0).. +(1,-0.5);

\draw (-2,0.3) ..controls +(-0.5,0)and +(0.5,0).. +(-1,0.5)  (-2,-0.3) ..controls +(-0.5,0)and +(0.5,0).. +(-1,-0.5);
\draw (-4,0.3) ..controls +(0.5,0)and +(-0.5,0).. +(1,0.5)  (-4,-0.3) ..controls +(0.5,0)and +(-0.5,0).. +(1,-0.5);

\draw (-4,0.3) ..controls +(-0.5,0)and +(0.5,0).. +(-1,0.5)  (-4,-0.3) ..controls +(-0.5,0)and +(0.5,0).. +(-1,-0.5);
\draw (-5,0.8) ..controls +(-0.7,0)and +(0,0.3).. +(-1,-0.8) (-5,-0.8) ..controls +(-0.7,0)and +(0,-0.3).. +(-1,0.8);

\draw (2,0.3) ..controls +(0.5,0)and +(-0.5,0).. +(1,0.5)  (2,-0.3) ..controls +(0.5,0)and +(-0.5,0).. +(1,-0.5);
\draw (3,0.8) ..controls +(0.7,0)and +(0,0.3).. +(1,-0.8) (3,-0.8) ..controls +(0.7,0)and +(0,-0.3).. +(1,0.8);

\draw (0,0) +(0,0.3)--+(0.1,0.3) +(0,-0.3)--+(0.1,-0.3);
\draw (2,0) +(0,0.3)--+(-0.1,0.3) +(0,-0.3)--+(-0.1,-0.3);

\draw[dashed] (0.3,0)--(1.7,0);

\draw (-1,-0.2) ..controls +(0.1,0).. +(0.3,0.2)  +(0,0) ..controls +(-0.1,0).. +(-0.3,0.2)
+(0,0.4) ..controls +(0.1,0).. +(0.3,0.2)  +(0,0.4) ..controls +(-0.1,0).. +(-0.3,0.2)
+(0.3,0.2)--+(0.4,0.3)  +(-0.3,0.2)--+(-0.4,0.3);

\draw (-3,-0.2) ..controls +(0.1,0).. +(0.3,0.2)  +(0,0) ..controls +(-0.1,0).. +(-0.3,0.2)
+(0,0.4) ..controls +(0.1,0).. +(0.3,0.2)  +(0,0.4) ..controls +(-0.1,0).. +(-0.3,0.2)
+(0.3,0.2)--+(0.4,0.3)  +(-0.3,0.2)--+(-0.4,0.3);

\draw (-5,-0.2) ..controls +(0.1,0).. +(0.3,0.2)  +(0,0) ..controls +(-0.1,0).. +(-0.3,0.2)
+(0,0.4) ..controls +(0.1,0).. +(0.3,0.2)  +(0,0.4) ..controls +(-0.1,0).. +(-0.3,0.2)
+(0.3,0.2)--+(0.4,0.3)  +(-0.3,0.2)--+(-0.4,0.3);

\draw (3,-0.2) ..controls +(0.1,0).. +(0.3,0.2)  +(0,0) ..controls +(-0.1,0).. +(-0.3,0.2)
+(0,0.4) ..controls +(0.1,0).. +(0.3,0.2)  +(0,0.4) ..controls +(-0.1,0).. +(-0.3,0.2)
+(0.3,0.2)--+(0.4,0.3)  +(-0.3,0.2)--+(-0.4,0.3);

\draw[dashed][red,very thick] (0,0) +(0,-0.3) ..controls +(0.1,0.3).. +(0,0.3);
\draw[red,very thick] (0,0) +(0,-0.3) ..controls +(-0.1,0.3).. +(0,0.3);

\draw[dashed][red,very thick] (2,0) +(0,-0.3) ..controls +(0.1,0.3).. +(0,0.3);
\draw[red,very thick] (2,0) +(0,-0.3) ..controls +(-0.1,0.3).. +(0,0.3);

\draw[dashed][red,very thick] (-2,0) +(0,-0.3) ..controls +(0.1,0.3).. +(0,0.3);
\draw[red,very thick] (-2,0) +(0,-0.3) ..controls +(-0.1,0.3).. +(0,0.3);

\draw[dashed][red,very thick] (-4,0) +(0,-0.3) ..controls +(0.1,0.3).. +(0,0.3);
\draw[red,very thick] (-4,0) +(0,-0.3) ..controls +(-0.1,0.3).. +(0,0.3);

\draw[dashed] (-1,0.5) +(0,-0.3) ..controls +(0.1,0.3).. +(0,0.3);
\draw (-1,0.5) +(0,-0.3) ..controls +(-0.1,0.3).. +(0,0.3);

\draw[dashed] (-1,-0.5) +(0,-0.3) ..controls +(0.1,0.3).. +(0,0.3);
\draw (-1,-0.5) +(0,-0.3) ..controls +(-0.1,0.3).. +(0,0.3);

\draw[dashed] (-3,0.5) +(0,-0.3) ..controls +(0.1,0.3).. +(0,0.3);
\draw (-3,0.5) +(0,-0.3) ..controls +(-0.1,0.3).. +(0,0.3);

\draw[dashed] (-3,-0.5) +(0,-0.3) ..controls +(0.1,0.3).. +(0,0.3);
\draw (-3,-0.5) +(0,-0.3) ..controls +(-0.1,0.3).. +(0,0.3);

\draw (-6,0) ..controls +(0.35,-0.1).. +(0.7,0);
\draw[dashed] (-6,0) ..controls +(0.35,0.1).. +(0.7,0);

\draw (4,0) ..controls +(-0.35,-0.1).. +(-0.7,0);
\draw[dashed] (4,0) ..controls +(-0.35,0.1).. +(-0.7,0);

\end{tikzpicture}
\end{center}
\caption{Surface $\mathcal{X}_g$} \label{pic7}
\end{figure}

The graph $G_g$ has $(3g-3)$ edges. The corresponding closed loops in $\mathcal{X}_g$ form a pants decomposition, denoted by $P$, of $\mathcal{X}_g$.

In the Fenchel-Nielsen  coordinate with respect to this pants decomposition $P$, the surface $\mathcal{X}_g$ has $(3g-3)$ length parameters being $\eps$ and arbitrary twist parameters. So we not only just construct one surface, actually we construct a subset of dimension $(3g-3)$ in moduli space $\sM_g$ in some sense. In this paper, we only use one such a surface to get the upper bound in Theorem \ref{mt-1}.

The (thick) closed geodesics in figures \ref{pic6} and \ref{pic7} are separating. They are $g_0$ ones in total. So in particular, we have the number
\bear
g_0 \leq \#\{\gamma \in P; \ \text{$\gamma$ is separating on $\mathcal{X}_g$}\}\leq (3g-3).
\eear

By \eqref{ver-eq} we know that
\bear \label{num-sep}
\#\{\gamma \in P; \ \text{$\gamma$ is separating on $\mathcal{X}_g$}\}\asymp_g g.
\eear

\begin{rem*}
These $g_0$ separating loops separate the surface $\mathcal{X}_g$ into $(g_0+1)$ components. Each component has similar area. More precisely, the largest area of the components is not bigger than twice that of the smallest one. Moreover, these components can be sorted from left to right. These are important in proving that the $k$-th eigenvalue of this surface $\mathcal{X}_g$ roughly realizes the rate of growth $\frac{1}{g^2}$ for large $g$. In Section \ref{sec-upp} we will see how these work in our proof.
\end{rem*}

Now we are ready to prove this surface $\mathcal{X}_g$ satisfies all those three properties in Proposition \ref{mt-2}. More precisely, 

\begin{pro*}[= Proposition \ref{mt-2}]
	For any constant $\eps>0$, there exists an integer $g(\eps)>0$, depending on $\eps$, such that for all $g\geq g(\eps)$, there exists a hyperbolic surface $\mathcal{X}_g $ of genus $g$ satisfying that
	
	$(a)$. $\mathcal{X}_g \in \sM_g^{\geq \eps}$.
	
	$(b)$. There exists a pants decomposition $P$ of $\mathcal{X}_g$ such that for every curve $\gamma\in P$, the length $\ell(\gamma)$ of $\gamma$ is equal to $\eps$.
	
	$(c)$. The number $\#\{\gamma \in P; \ \text{$\gamma$ is separating on $\mathcal{X}_g$}\}\asymp_g g$.
\end{pro*}

\bp
Let $\mathcal{X}_g$ be the closed hyperbolic surface of genus $g$ as constructed above.

First by the construction above, the obvious pants decomposition $P$ clearly satisfies part $(b)$.

By \eqref{num-sep} we also know that part $(c)$ holds.

So it suffices to show part $(a)$. We will show that
\[\sys(\mathcal{X}_g) = \eps.\]

Let $\gamma$ be a shortest non-trivial loop on $\mathcal{X}_g$. Since every curve in $P$ has length equal to $\eps$, we have
\[\ell(\gamma)\leq \eps.\]

If $\gamma \in P$, then we are done.

Now we assume that $\gamma \notin P$. Since $P$ is a pants decomposition of $\mathcal{X}_g$, one may assume that $\gamma$ (transversely) intersects with certain closed geodesics in $P$. This in particular induces a closed loop, denoted by $\Gamma$, on the corresponding trivalent graph $G_g$. One may see it as follows: as $\gamma$ goes from one pair of pants to its neighboring one through a boundary curve, $\Gamma$ goes from the corresponding vertex to the corresponding neighboring vertex through the corresponding edge, and finally $\Gamma$ will come back to its original vertex.

We will split the remaining proof into two cases.

Case-1: The closed loop $\Gamma \subset G_g$ is non-trivial.

For this case, $\Gamma$ contains a non-trivial simple loop in $G_g$. By \eqref{gir-in-g} we know that $$\mathop{\rm girth}(G_g)\geq W(\eps).$$

So we have,
\[\length{\Gamma} \geq W(\eps).\]

Recall that on each pair of pant $\mathcal{P}_\eps$, the distance between two boundary curves is $d(\eps)$ (as shown in figure \ref{pic2}). So we have
$$\ell(\gamma) \geq W(\eps)\cdot d(\eps).$$

Then by \eqref{w-s-e} we have
\[\ell(\gamma) \geq 2\eps.\]

Which is a contradiction.\\

Case-2: The closed loop $\Gamma \subset G_g$ is trivial.

For this case, the closed loop $\Gamma$ must turn back twice in the sense that it goes through one edge and then immediately comes back through this edge. For turning back, the corresponding arc $\tau'$ in $\gamma$ will be contained in one precise pair of pants $\mathcal{O}$. And it starts from a boundary curve in this pair of pants and then comes back to this boundary curve again. Since $\gamma$ is a closed geodesic, this arc $\tau'$ can not be homotopic to certain arc in the boundary curve. And this arc $\tau'$ has two intersection points with one boundary curve of $\mathcal{O}$. So the length $\ell(\tau')$ of this arc must be bigger than or equal to the length of the arc $\tau$ as shown in figure \ref{pic2}. By \eqref{tau-in} we know that
\[\ell(\tau')\geq \ell(\tau) > \frac{\eps}{2}.\]

Since $\gamma$ is a closed geodesic and $\Gamma \subset G_g$ is trivial, the complement $\gamma \setminus \tau'$ of $\tau'$ also contains an arc $\tau''$ which has two intersection points with one boundary curve of another pair of pants. Similar as above we also have
\[\ell(\tau'')\geq \ell(\tau) > \frac{\eps}{2}.\]

Thus, we have
\beqar
\ell(\gamma) &\geq& \ell(\tau')+\ell(\tau'')\\
&>&\eps.
\eeqar

This again contradicts our assumption that $\gamma$ is a shortest non-trivial loop.\\

Thus, only all the curves in $P$ realize the systole of $\mathcal{X}_g$. So we have
$$\sys(\mathcal{X}_g) = \eps.$$

In particular, part (a) also holds. We finish the proof.
\ep

%%%%%%%%%%%%%%%%%%%%%%%%%%%%%%%%%%%%%%%%%%%%%%%%%%%
\section{Upper Bound}\label{sec-upp}
In this section we will prove the upper bounds in Theorem \ref{mt-1} and \ref{mt-large}. Let $\mathcal{X}_g$ be a closed surface of genus $g$ in Proposition \ref{mt-2} as constructed in the last section. Recall that $\sys(\mathcal{X}_g)=\eps$. We will show that
\begin{proposition} \label{upp}
For all integer $k\geq 1$ there exists a positive constant $\beta(\eps)$, only depending on $\eps$, such that for large enough $g$,
\[\lambda_k(\mathcal{X}_g) \leq \frac{\beta(\eps) \cdot k^2}{g^2}.\]
\end{proposition}

By construction in last section, in figure \ref{pic6} (or figure \ref{pic7} for small $\eps$), we let $\{\gamma_i\}_{1\leq i \leq g_0}$ be these $g_0$ separating (thick) closed geodesics, which are listed in order from left to right. The union $\cup_{1\leq i \leq g_0} \gamma_i$ separates the surface $\mathcal{X}_g$ into $(g_0+1)$ components. We denote them by $M_0,M_1,...,M_{g_0-1}, M_{g_0}$ in order from left to right. For each $1\leq i \leq g_0-1$, each component $M_i$ is a hyperbolic surface of genus $\frac{2+V_0}{2}$ with $2$ closed geodesic boundaries. So the Euler characteristic
\[\chi(M_i)=2-2\cdot\frac{2+V_0}{2}-2=-2-V_0.\]

By the Gauss-Bonnet formula we know that for each $1\leq i \leq g_0-1$,
\bear\label{Area-i}
\area(M_i)=2\pi(V_0+2).
\eear

For $1\leq i \leq g_0$, we let $K_i$ be the collar of $\gamma_i$ (as defined in the Collar Lemma: see Lemma \ref{collar}). For each collar $K_i$, we let $w(\eps)$ to be its half width. That is, $$w(\eps)=\mathop{\rm arcsinh} \frac{1}{\sinh \frac{1}{2}\eps}.$$

Recall by Lemma \ref{collar} the hyperbolic metric on $K_i$ is
\[ds^2=d\rho^2+\eps^2 \cosh^2{\rho}dt^2.\]

Thus, the area of each collar is
\bear \label{are-k}
\area(K_i) & = & \int_{K_i} dV \\
           & = & \int_0^1 \int_{-w}^w \eps \cosh \rho d\rho dt \nonumber \\
           & = & 2\eps \sinh w  \nonumber \\
           & = & \frac{2\eps}{\sinh \frac{1}{2}\eps}. \nonumber
\eear

It is clear that $\area(K_i)<4$.

For each $1\leq i \leq g_0-1$, we set
$$M_i^0 = M_i-\cup_{j=1}^{g_0} K_j$$
to be the complement. Thus, for each $1\leq i \leq g_0-1$ the area satisfies that
\bear \label{are-c-k}
\area(M_i^0) & = & 2\pi(V_0+2) - \frac{1}{2}\area(K_i)-\frac{1}{2}\area(K_{i+1}) \\
             & = & 2\pi(V_0+2) - \frac{2\eps}{\sinh \frac{1}{2}\eps} \nonumber
\eear

It is clear that $\area(M_i^0)>2\pi(V_0+2) - 4>0$.\\

Now we are ready to prove Proposition \ref{upp}.
\bp [Proof of Proposition \ref{upp}]
For any fixed integer $k>0$, let $g_1>0$ be an integer with
\bear \label{g_0-eq}
g_0 -1 = g_1 (k+1) +r
\eear
where $0\leq r\leq k$.

For each $0 \leq i \leq k$ and $1\leq j \leq g_1$, we set
$$M^0_{i,j} = M^0_{ig_1 +j}$$
and
$$A_i = \cup_{1\leq j \leq g_1} M^0_{i,j}.$$
In this way, we divide the union $\cup_{1\leq l \leq g_0-1} M^0_l$ into $(k+1)$ piles $A_0,A_1,...,A_k$ and $r$ remains. For each pile $A_i$, it has $g_1(\asymp_g g)$ components.

For $i=0,1,...,k$, we define a function $\varphi_i$ on $\mathcal{X}_g$ as follows.

Set $\varphi_i = 0$ on $M_j$ for those $M_j\cap A_i=\emptyset$.

On $M_{i,j}^0 \subset A_i$, set
\begin{equation}
 \varphi_i = \left\{
\begin{array}{ll}
 0 & \text{on $M_{i,1}^0 \cup M_{i,g_1}^0$}, \\
 1 & \text{on $M_{i,2}^0 \cup M_{i,g_1-1}^0$}, \\
 2 & \text{on $M_{i,3}^0 \cup M_{i,g_1-2}^0$}, \\
 3 & \text{on $M_{i,4}^0 \cup M_{i,g_1-3}^0$}, \\
\cdots \\
\frac{g_1}{2}-1 & \text{on $M_{i,g_1 /2}^0\cup M_{i,(g_1 /2+1)}^0$ if $g_1$ is even}, \\
\frac{g_1-1}{2} & \text{on $M_{i,(g_1+1) /2}^0$ if $g_1$ is odd}. \\
\end{array}\right. .
\end{equation}

Now we have defined $\varphi_i$ on each $M_j^0$. It remains to define $\varphi_i$ on the collars of these $g_0$ separating (thick) closed geodesics (as shown in figure \ref{pic6} or \ref{pic7}). We define $\varphi_i$ as shown in the following picture.

\begin{figure}[ht]
\begin{center}
\begin{tikzpicture}[scale=0.8]

\draw (0,0) +(0,1) ..controls+(-1,0)and+(0.8,-0.5).. +(-3,1.7)
+(0,1) ..controls+(1,0)and+(-0.8,-0.5).. +(3,1.7)
+(0,-1) ..controls+(-1,0)and+(0.8,0.5).. +(-3,-1.7)
+(0,-1) ..controls+(1,0)and+(-0.8,0.5).. +(3,-1.7);

\draw (-3,1.7)--(-3.8,2.2) (-3,-1.7)--(-3.8,-2.2) (3,1.7)--(3.8,2.2) (3,-1.7)--(3.8,-2.2);

\draw[red, very thick] (0,0) +(0,1) ..controls+(-0.2,-1).. +(0,-1);
\draw[red, very thick][dashed] (0,0) +(0,1) ..controls+(0.2,-1).. +(0,-1);

\draw (-3,0) +(0,1.7) ..controls+(-0.4,-1.7).. +(0,-1.7);
\draw[dashed] (-3,0) +(0,1.7) ..controls+(0.4,-1.7).. +(0,-1.7);

\draw (3,0) +(0,1.7) ..controls+(-0.4,-1.7).. +(0,-1.7);
\draw (3,0) +(0,1.7) ..controls+(0.4,-1.7).. +(0,-1.7);

\draw (-4.5,0) node{$M_l^0$};
\draw (4.5,0) node{$M_{l+1}^0$};

\draw (-4,-2.5) node {$\rho=$};
\draw (-3,-2.5) node {$-w$};
\draw (0,-2.5) node {$0$};
\draw (3,-2.5) node {$w$};

\draw (-4,-3.5) node {$\varphi_i=$};
\draw (-3,-3.5) node {$a$};
\draw (0,-3.5) node {$\frac{a+b}{2}$};
\draw (3,-3.5) node {$b$};

\draw[->] (-2,1.8)--(2,1.8);
\draw (0,2.1) node {$\rho$};

\draw[->] (-0.3,-0.9) ..controls+(-0.2,0.9).. (-0.3,0.9);
\draw (-0.7,0) node {$t$};

\end{tikzpicture}
\end{center}
\caption{} \label{pic8}
\end{figure}

For each collar, we use polar coordinate $(\rho,t)\in [-w,w] \times \mathbb S ^1$ where $w=w(\eps)=\mathop{\rm arcsinh} \frac{1}{\sinh \frac{1}{2}\eps}$ is the half width (see Lemma \ref{collar}). As figure \ref{pic8} shows, the two boundary cycles are parts of boundaries of two consecutive components $M_l^0$ and $M_{l+1}^0$ for some $0\leq l \leq g_0$. So $\varphi_i$ takes constant values $a$ and $b$ on each boundary loop respectively. We linearly extend $\varphi_i$ onto the collar in the following sense
$$\varphi_i(\rho,t) = \varphi_i(\rho) = \frac{a+b}{2} + \frac{b-a}{2\arctan(\tanh\frac{w}{2})}\arctan(\tanh\frac{\rho}{2})$$
for every $(\rho,t)\in [-w,w] \times \mathbb S ^1$.

In this way, for each $0\leq i \leq k$ we have defined a $H^1$ function $\varphi_i$ on $\mathcal{X}_g$. Now we make estimations of $\varphi_i$ on $\mathcal{X}_g$.

First by definition, for each $0\leq i \leq k$ we have
\bear
\int_{\mathcal{X}_g} \varphi_i^2 & > & \int_{\bigcup_{1\leq j \leq [\frac{g_1+1}{2}]} M_{i,j}^0} \varphi_i^2 \\
&=& \sum_{j=1}^{[\frac{g_1+1}{2}]} (j-1)^2 \cdot \area(M_{i,j}^0) \nonumber  \\
                       & = & \sum_{j=1}^{[\frac{g_1+1}{2}]} (j-1)^2\cdot (2\pi(V_0+2) - \frac{2\eps}{\sinh \frac{1}{2}\eps}) \nonumber
\eear
where we apply \eqref{are-c-k} in the last equation.

For all integer $n\geq 1$ the following elementary equation is well-known.
\[\sum_{j=1}^nj^2=\frac{n(n+1)(2n+1)}{6}\geq \frac{n^3}{3}.\]

Thus, for each $0\leq i \leq k$ we have
\bear \label{low-g^3}
\int_{\mathcal{X}_g} \varphi_i^2  >(2\pi(V_0+2) - \frac{2\eps}{\sinh \frac{1}{2}\eps})\cdot \frac{(g_1-2)^3}{24}.
\eear

On each collar where $\varphi_i$ is not a constant map, the function $\varphi_i$ takes values of difference equal to $1$ on the two boundaries curves of the collar. By the definition of $\varphi$ we have at most $(g_1-1)$ such collars (if $g_1$ is even, we only have $(g_1-2)$ such collars). We denote them by $K_l'$. So on each $K_l'$,
$$\varphi_i= c_l \pm \frac{1}{2\arctan(\tanh\frac{w}{2})}\arctan(\tanh\frac{\rho}{2})$$
for some constant $c_l$.

Recall the hyperbolic metric $ds^2=d\rho^2 + \eps^2 \cosh^2\rho dt^2$ on each collar (see Lemma \ref{collar}). For a smooth function $f(\rho,t)$ on each collar, the gradient $\nabla f$ of $f$ is
$$\nabla f= \frac{\partial f}{\partial \rho}\frac{\partial }{\partial \rho} + \frac{1}{l^2 \cosh^2\rho}\frac{\partial f}{\partial t}\frac{\partial }{\partial t}.$$

So on each collar the energy density of $\varphi_i$ is
$$|\nabla \varphi_i|^2(\rho,t)= |\frac{1}{2\arctan(\tanh\frac{w}{2})} \frac{1}{2\cosh\rho}|^2.$$

Recall that $\varphi_i$ is constant on $\mathcal{X}_g \setminus \cup_{l=1}^{g_1-1}K'_l$. In particular, $\nabla \varphi_i \equiv 0$ on $\mathcal{X}_g \setminus \cup_{l=1}^{g_1-1}K'_l$. Thus, for each $0\leq i \leq k$ we have,
\bear \label{upp-g}
\int_{\mathcal{X}_g} |\nabla\varphi_i|^2 & \leq & \sum_{l=1}^{g_1-1} \int_{K_l'} |\nabla(\frac{1}{2\arctan(\tanh\frac{w}{2})}\arctan(\tanh\frac{\rho}{2}))|^2 \\
                               & = & (g_1-1)\int_0^1\int_{-w}^w |\frac{1}{2\arctan(\tanh\frac{w}{2})} \frac{1}{2\cosh\rho}|^2 \eps \cosh\rho \ d\rho dt \nonumber \\
                               & = & (g_1-1)\frac{1}{(2\arctan(\tanh\frac{w}{2}))^2}\eps \arctan(\tanh\frac{w}{2}) \nonumber\\
                               & \leq & g_1\cdot \frac{\eps}{4\arctan(\tanh\frac{w}{2})} \nonumber
\eear

By \eqref{low-g^3} and \eqref{upp-g}, for each $0\leq i \leq k$ we have
\bear \label{upp-frac}
\frac{\int_{\mathcal{X}_g} \abs{\nabla \varphi_i}^2}{\int_{\mathcal{X}_g} \varphi_i^2} \leq \frac{24\eps}{4\arctan(\tanh\frac{w}{2}) \cdot (2\pi(V_0+2) - \frac{2\eps}{\sinh \frac{1}{2}\eps})}\cdot \frac{g_1}{(g_1-2)^3}.
\eear

By \eqref{ver-eq} and \eqref{g_0-eq} we know that
\[\lim_{g \to \infty} \frac{g_1}{g}=\frac{2}{(k+1)(2+V_0)}\]
where $V_0$ only depends on $\eps$. Actually the following inequality always holds:
\[\frac{g_1}{g}\geq \frac{1}{24V_0k}.\]

Thus, for each $0\leq i \leq k$ we have that for large enough $g$ ,
\bear \label{upp-i}
\frac{\int_{\mathcal{X}_g} \abs{\nabla \varphi_i}^2}{\int_{\mathcal{X}_g} \varphi_i^2} \leq \beta(\eps)\cdot \frac{k^2}{g^2}
\eear
where
\bear \label{beta-eq}
\beta(\eps)=\frac{10^{10}\eps V_0^2}{\arctan(\tanh\frac{w}{2}) \cdot (\pi(V_0+2) - \frac{\eps}{\sinh \frac{1}{2}\eps})}.
\eear

Recall that our construction of $\varphi_i$ ensures that the supports satisfy
\bear \label{empty}
supp(\varphi_i)\cap supp(\varphi_j) = \emptyset, \quad \forall 0\leq i\neq j \leq k.
\eear

In particular, these functions $\varphi_0,\varphi_1,...,\varphi_k$ are linearly independent. So
\[\dim(span\{\varphi_i\}_{0\leq i \leq k})=k+1.\]

One may choose a non-constant function
$$\varphi=a_0\varphi_0+...+a_k\varphi_k$$
for certain constants $a_i \ (0\leq i \leq k)$ such that
\[\int_{\mathcal{X}_g} \varphi \cdot f_l=0, \quad \forall l \in [0,k-1]\]
where $f_l$ is the normalized $l$-th eigenfunction of $\mathcal{X}_g$.

Recall equation \eqref{lamb-k} says that the $k$-th eigenvalue satisfies that
\[\lambda_k(\mathcal{X}_g)=\inf_{}\{ \frac{\int_{\mathcal{X}_g}|\nabla f|^2}{\int_{\mathcal{X}_g}f^2}; \ f \in H^1 \ \text{and} \ \int_{\mathcal{X}_g}f\cdot f_l=0 \ \forall l \in [0, k-1]\}.\]

Therefore, we have that for large enough $g$,
\bear
\lambda_k(\mathcal{X}_g) &\leq& \frac{\int_{\mathcal{X}_g} \abs{\nabla \varphi}^2}{\int_{\mathcal{X}_g} \varphi^2} \\
&=& \frac{\sum_{i=0}^ka_i^2\cdot \int_{\mathcal{X}_g} \abs{\nabla \varphi_i}^2}{\sum_{i=0}^k a_i^2\cdot{ \int_{\mathcal{X}_g} \varphi_i^2}} \quad \quad \text{(by \eqref{empty})}\nonumber \\
&\leq & \beta(\eps)\cdot \frac{k^2}{g^2}. \quad \quad \quad \quad \quad \quad \quad  (\text{by \eqref{upp-i}}) \nonumber
\eear
\
If the constant $\eps$ satisfies $$0<\eps< 2\arcsinh 1 = 2\ln (1+\sqrt 2),$$
then the hyperbolic surface $\mathcal{X}_g$ may be shown as in figure \ref{pic7}. For this case, the constant $d=d(\eps)>\eps$. As in Remark \ref{r-e-small} one may take $V_0=1$. So the constant $\beta(\eps)$ (see \eqref{beta-eq}) in Proposition \ref{upp} can be taken to be
\[\beta(\eps)=c\cdot \eps\]
where $c>0$ is a universal constant.

The proof is complete.
\ep

\bp [Proof of Theorem \ref{mt-1}]
It clearly follows by Proposition \ref{low} and \ref{upp}.
\ep

\begin{rem*}
By adding one more edge between the left subgraph $G(V_0)$ and the right subgraph $G(V_1)$ in $G_g$ as shown in Figure \ref{pic4}, we get a new trivalent graph denoted by $G'_{g+1}$. Since we have two more vertices, the new corresponding hyperbolic surface $\mathcal{X}'_{g+1}$ has genus of $(g+1)$. The proof of Proposition \ref{mt-2} can also give that $$ \mathcal{X}'_{g+1} \in \sM_{g+1}^{\geq \eps}.$$
Moreover, the proof of Theorem \ref{mt-1} also yields that for any integer $k>0$,
\[\lambda_k(\mathcal{X}'_g)\asymp_g \frac{1}{g^2}.\]

\begin{figure}[ht]
	\begin{center}
		\begin{tikzpicture}[scale=1]
		
		\draw(0.7,0.7)..controls(1,0.7)and(1.5,0.3)..(1.7,0.4)..controls(1.9,0.5)and(1.9,1)..(2,1.3);
		\draw(0.7,-0.7)..controls(1.3,-0.7)and(1.6,0)..(1.9,0.1)..controls(3.1,0.3)and(3.2,0.7)..(3.3,1);
		\draw(0.7,0.3)..controls(1.3,0.3)and(1.3,-0.2)..(0.7,-0.3);
		\draw(2.4,1.2)..controls(2.1,0.5)and(2.6,0.4)..(2.9,1.1);
		
		\draw(0.7,0.7)..controls(0.65,0.5)..(0.7,0.3);
		\draw(0.7,0.7)..controls(0.75,0.5)..(0.7,0.3);
		\draw(0.7,-0.7)..controls(0.65,-0.5)..(0.7,-0.3);
		\draw(0.7,-0.7)..controls(0.75,-0.5)..(0.7,-0.3);
		\draw[red,very thick] (1.7,0.4)..controls(1.73,0.2)..(1.9,0.1);
		\draw[red,very thick,dashed] (1.7,0.4)..controls(1.87,0.3)..(1.9,0.1);
		\draw(2,1.3)..controls(2.2,1.3)..(2.4,1.2);
		\draw(2,1.3)..controls(2.2,1.2)..(2.4,1.2);
		\draw(3.3,1)..controls(3.1,1.1)..(2.9,1.1);
		\draw(3.3,1)..controls(3.1,1)..(2.9,1.1);
		
		\draw(2.2,2.2)..controls(2.4,2.7)and(2.8,2.6)..(3,3);
		\draw(3.5,1.9)..controls(3.6,2.4)and(3.3,2.4)..(3.3,2.9);
		\draw(2.6,2.1)..controls(2.8,2.3)and(3.1,2.7)..(3.1,2);
		
		\draw(2.2,2.2)..controls(2.4,2.1)..(2.6,2.1);
		\draw(2.2,2.2)..controls(2.4,2.2)..(2.6,2.1);
		\draw(3.5,1.9)..controls(3.3,1.9)..(3.1,2);
		\draw(3.5,1.9)..controls(3.3,2)..(3.1,2);
		\draw[red,very thick] (3,3)..controls(3.2,3)..(3.3,2.9);
		\draw[red,very thick] (3,3)..controls(3.1,2.9)..(3.3,2.9);

		\draw(-0.7,0.7)..controls(-1,0.7)and(-1.5,0.3)..(-1.7,0.4)..controls(-1.9,0.5)and(-1.9,1)..(-2,1.3);
		\draw(-0.7,-0.7)..controls(-1.3,-0.7)and(-1.6,0)..(-1.9,0.1)..controls(-3.1,0.3)and(-3.2,0.7)..(-3.3,1);
		\draw(-0.7,0.3)..controls(-1.3,0.3)and(-1.3,-0.2)..(-0.7,-0.3);
		\draw(-2.4,1.2)..controls(-2.1,0.5)and(-2.6,0.4)..(-2.9,1.1);
		
		\draw(-0.7,0.7)..controls(-0.65,0.5)..(-0.7,0.3);
		\draw(-0.7,0.7)..controls(-0.75,0.5)..(-0.7,0.3);
		\draw(-0.7,-0.7)..controls(-0.65,-0.5)..(-0.7,-0.3);
		\draw(-0.7,-0.7)..controls(-0.75,-0.5)..(-0.7,-0.3);
		\draw[red,very thick,dashed] (-1.7,0.4)..controls(-1.73,0.2)..(-1.9,0.1);
		\draw[red,very thick] (-1.7,0.4)..controls(-1.87,0.3)..(-1.9,0.1);
		\draw(-2,1.3)..controls(-2.2,1.3)..(-2.4,1.2);
		\draw(-2,1.3)..controls(-2.2,1.2)..(-2.4,1.2);
		\draw(-3.3,1)..controls(-3.1,1.1)..(-2.9,1.1);
		\draw(-3.3,1)..controls(-3.1,1)..(-2.9,1.1);
			
		\draw(-2.2,2.2)..controls(-2.4,2.7)and(-2.8,2.6)..(-3,3);
		\draw(-3.5,1.9)..controls(-3.6,2.4)and(-3.3,2.4)..(-3.3,2.9);
		\draw(-2.6,2.1)..controls(-2.8,2.3)and(-3.1,2.7)..(-3.1,2);
		
		\draw(-2.2,2.2)..controls(-2.4,2.1)..(-2.6,2.1);
		\draw(-2.2,2.2)..controls(-2.4,2.2)..(-2.6,2.1);
		\draw(-3.5,1.9)..controls(-3.3,1.9)..(-3.1,2);
		\draw(-3.5,1.9)..controls(-3.3,2)..(-3.1,2);
		\draw[red,very thick] (-3,3)..controls(-3.2,3)..(-3.3,2.9);
		\draw[red,very thick] (-3,3)..controls(-3.1,2.9)..(-3.3,2.9);

		\draw[dashed](-0.6,0.5)..controls(0,0.7)..(0.6,0.5);
		\draw[dashed](-0.6,-0.5)..controls(0,-0.7)..(0.6,-0.5);
		\draw[dashed](2.2,1.3)..controls(2.2,1.7)..(2.4,2.1);
		\draw[dashed](3.1,1.1)..controls(3.3,1.5)..(3.3,1.9);
		\draw[dashed](-2.2,1.3)..controls(-2.2,1.7)..(-2.4,2.1);
		\draw[dashed](-3.1,1.1)..controls(-3.3,1.5)..(-3.3,1.9);

		\draw(0.7,5.3)..controls(1,5.3)and(1.5,5.7)..(1.7,5.6);
		\draw(0.7,6.7)..controls(1.3,6.7)and(1.6,6)..(1.9,5.9);
		\draw(0.7,5.7)..controls(1.3,5.7)and(1.3,6.2)..(0.7,6.3);
		
		\draw(0.7,5.3)..controls(0.65,5.5)..(0.7,5.7);
		\draw(0.7,5.3)..controls(0.75,5.5)..(0.7,5.7);
		\draw(0.7,6.7)..controls(0.65,6.5)..(0.7,6.3);
		\draw(0.7,6.7)..controls(0.75,6.5)..(0.7,6.3);
		\draw[red,very thick] (1.7,5.6)..controls(1.73,5.8)..(1.9,5.9);
		\draw[red,very thick] (1.7,5.6)..controls(1.87,5.7)..(1.9,5.9);
		
		\draw(-0.7,5.3)..controls(-1,5.3)and(-1.5,5.7)..(-1.7,5.6);
		\draw(-0.7,6.7)..controls(-1.3,6.7)and(-1.6,6)..(-1.9,5.9);
		\draw(-0.7,5.7)..controls(-1.3,5.7)and(-1.3,6.2)..(-0.7,6.3);
		
		\draw(-0.7,5.3)..controls(-0.65,5.5)..(-0.7,5.7);
		\draw(-0.7,5.3)..controls(-0.75,5.5)..(-0.7,5.7);
		\draw(-0.7,6.7)..controls(-0.65,6.5)..(-0.7,6.3);
		\draw(-0.7,6.7)..controls(-0.75,6.5)..(-0.7,6.3);
		\draw[red,very thick] (-1.7,5.6)..controls(-1.73,5.8)..(-1.9,5.9);
		\draw[red,very thick] (-1.7,5.6)..controls(-1.87,5.7)..(-1.9,5.9);
		
		\draw[dashed](-0.6,5.5)..controls(0,5.3)..(0.6,5.5);
		\draw[dashed](-0.6,6.5)..controls(0,6.7)..(0.6,6.5);

		\draw[dashed](1.9,5.75)..controls(2.9,4.7)..(3.15,3);
		\draw[dashed](-1.9,5.75)..controls(-2.9,4.7)..(-3.15,3);
		
		\draw (0,0) node {$\Sigma(G(V_1))$};
		\draw (0,6) node {$\Sigma(G(V_0))$};
		\draw (2.7,1.6) node {$\Sigma(G(V_0))$};
		\draw (-2.8,1.6) node {$\Sigma(G(V_0))$};
		\end{tikzpicture}
	\end{center}
	\caption{surface $\mathcal{X}'_{g+1}$} \label{pic9}
\end{figure}
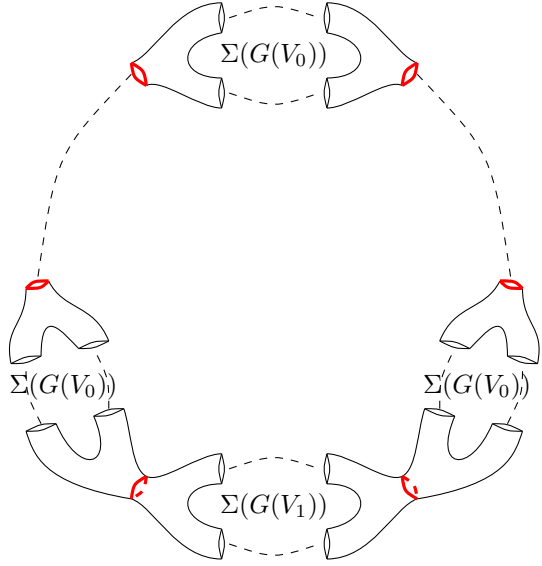
\noindent In this case, we remark that the surface $\mathcal{X}'_g$ does not satisfy that \textsl{the number of separating systolic curves on $\mathcal{X}'_g$ is uniformly comparable to the genus $g$ of $\mathcal{X}'_g$.} Actually in the proof of Proposition \ref{upp}, the essential point is existence of a collection of systolic curves, which has number uniformly comparable to the genus $g$ and can be ordered from left to right in some sense.
\end{rem*}

\begin{rem*}
The reason for choosing the function $\arctan(\tanh\frac{\rho}{2})$ to extend $\varphi_i$ on collars is that this makes $\int_{\text{one collar}} |\nabla\varphi_i|^2$ to be minimal (by fixing constant boundary values). For the proof of Proposition \ref{upp}, this choice is not essential. One may also choose some other function $\varphi_i$ to make $\int_{\text{collar}} |\nabla\varphi_i|^2$ to be comparable to $1$. It will also yield the uniform upper bound in Theorem \ref{mt-1}.

The following lemma may tell that why the function $\arctan(\tanh\frac{\rho}{2})$ makes the energy $\int_{\text{one collar}} |\nabla\varphi_i|^2$ to be minimal.
\begin{lemma}\label{f(collar)}
For a collar $K$ with center geodesic length $\ell$ and half width $w$ and using the coordinate $(\rho,t)$ defined in Lemma \ref{collar}, we have for two given constants $a$ and $b$,
\begin{equation}\label{min(int|nabla f|^2)}
\min \int_K|\nabla f|^2 = (a-b)^2 \frac{l}{4\arctan (\tanh \frac{w}{2})}
\end{equation}
where the minimum is taken over all functions $f(\rho,t)$ with $$f|_{\{\rho=-w\}}=a\ \text{and} \ f|_{\{\rho=w\}}=b.$$

The equality in \eqref{min(int|nabla f|^2)} holds if and only if
\begin{equation}\label{defn f(collar)}
f(\rho,t) = \frac{a+b}{2} + \frac{b-a}{2\arctan(\tanh\frac{w}{2})}\arctan(\tanh\frac{\rho}{2}).
\end{equation}
\end{lemma}

\begin{proof}
We use variation method.

Recall that the hyperbolic metric on the collar is
$$ds^2 = d\rho^2 + l^2\cosh^2 \rho dt^2.$$
So for any function $f$ we have
$$\nabla f = \frac{\partial f}{\partial \rho}\frac{\partial}{\partial \rho} + \frac{1}{l^2\cosh^2\rho}\frac{\partial f}{\partial t}\frac{\partial f}{\partial t},$$
$$\abs{\nabla f}^2 = (\frac{\partial f}{\partial \rho})^2 + \frac{1}{l^2 \cosh^2 \rho}(\frac{\partial f}{\partial t})^2,$$
$$\int_K \abs{\nabla f}^2 = \int_0^1 \int_{-w}^w l\cosh \rho\  (\frac{\partial f}{\partial \rho})^2 + \frac{1}{l \cosh \rho}(\frac{\partial f}{\partial t})^2 d\rho dt.$$

For fixed $t$, let
$$A_t(f) := \int_{-w}^w \cosh \rho\  (\frac{\partial f}{\partial \rho})^2 d\rho.$$

Let $h(\rho)$ be an arbitrary smooth function on $[-w,w]$ such that $h(-w)=h(w)=0$. Then the first derivative is
\begin{eqnarray*}
\frac{d}{ds}|_{s=0} (\int_{-w}^w \cosh \rho\  (\frac{\partial (f+sh)}{\partial \rho})^2 d\rho)
& = & \int_{-w}^w \cosh \rho\  2\frac{\partial f}{\partial \rho}\frac{\partial h}{\partial \rho} d\rho \\
& = & -2\int_{-w}^w \frac{\partial}{\partial \rho}(\cosh \rho\  \frac{\partial f}{\partial \rho}) h d\rho
\end{eqnarray*}

Since $h$ is arbitrary, $A_t(f)$ takes the minimum only if
$$\frac{\partial}{\partial \rho}(\cosh \rho\  \frac{\partial f}{\partial \rho}) = 0.$$

Since $f|_{\{\rho=-w\}}=a$ and $f|_{\{\rho=w\}}=b$, by elementary ODE theory we know that
$$f(\rho) = \frac{a+b}{2} + \frac{b-a}{2\arctan(\tanh\frac{w}{2})}\arctan(\tanh\frac{\rho}{2}).$$

Thus, for any $f$ we have
\begin{eqnarray*}
A_t(f)
& = & \int_{-w}^w \cosh \rho\  (\frac{\partial f}{\partial \rho})^2 d\rho \\
& \geq & \int_{-w}^w \cosh \rho\  (\frac{\partial (\frac{a+b}{2} + \frac{b-a}{2\arctan(\tanh\frac{w}{2})}\arctan(\tanh\frac{\rho}{2}))}{\partial \rho})^2 d\rho \\
& = & \int_{-w}^w (\frac{b-a}{4\arctan(\tanh\frac{w}{2})})^2 \frac{1}{\cosh \rho} d\rho \\
& = & \frac{(b-a)^2}{4\arctan(\tanh\frac{w}{2})}.
\end{eqnarray*}

Therefore, for any $f$ we have
\begin{eqnarray*}
\int_K \abs{\nabla f}^2
& = & \int_0^1 l A_t(f) dt + \int_0^1 \int_{-w}^w \frac{1}{l \cosh \rho}(\frac{\partial f}{\partial t})^2 d\rho dt \\
& \geq & \int_0^1 l A_t(f) dt \\
& \geq & \frac{(b-a)^2l}{4\arctan(\tanh\frac{w}{2})}.
\end{eqnarray*}
It is clear that the equality holds if and only if
\begin{equation*}
f(\rho,t) = \frac{a+b}{2} + \frac{b-a}{2\arctan(\tanh\frac{w}{2})}\arctan(\tanh\frac{\rho}{2}).
\end{equation*}
Which completes the proof.
\end{proof}
\end{rem*}

Now we start to prove Theorem \ref{mt-large}.

\begin{proof} [Proof of Theorem \ref{mt-large}]
Let $\mathcal{X}_g \in \sM_g^{\geq \eps}$ be a hyperbolic surface in Proposition \ref{mt-2} which is constructed in Section \ref{sec-ex}. By construction, the (bold) closed geodesics in figure \ref{pic6} and \ref{pic7} are separating and can be sorted from left to right. Each of these separating geodesics has length $\eps$ and hence admits an embedded collar with half width 
$$\omega(\eps)=\mathop{\rm arcsinh} \frac{1}{\sinh \frac{1}{2}\eps}.$$
Clearly the number of these separating geodesics is at least $c(\eps)\cdot(g-1)$ for some constant $c(\eps)>0$ only depending on $\eps$. Here if $\eps<2\mathop{\rm arcsinh}1$ (see figure \ref{pic7}), then $c(\eps)=1$ and the number is exactly $(g-1)$. Hence we have that the surface $\mathcal{X}_g$ has diameter 
$$\mathrm{diam}(\mathcal{X}_g) \geq 2\omega(\eps) c(\eps)  (g-1).$$
	
\noindent Recall that by Cheng \cite[Corollary 2.3]{Che75} we have that for all $k> 0$, 
$$\lambda_k(\mathcal{X}_g) \leq \frac{1}{4} + \frac{16\pi^2k^2}{\mathrm{diam}^2(\mathcal{X}_g)}.$$ 
Now we take $k=ag$, then 
$$\min\limits_{X\in \M_g^{\geq \eps}} \lambda_{ag}(X) \leq \lambda_{ag}(\mathcal{X}_g) \leq \frac{1}{4} + a^2 \cdot \frac{5\pi^2}{\omega(\eps)^2 c(\eps)^2}.$$
Then the upper bound follows by choosing
\[\theta(\eps)=  \frac{5\pi^2}{\omega(\eps)^2 c(\eps)^2}.\] 
For small $\eps>0$, $c(\eps)=1$ and $\omega(\eps)$ is large. So clearly we have $\theta(\eps)\to 0$ as $\eps\to 0$.
\end{proof}

\bibliographystyle{plain}
\bibliography{ref}

\end{document}